\documentclass[12pt,a4paper,english,draft]{smfart}
\usepackage{amssymb,amsmath,stmaryrd,vmargin,smfthm}
\usepackage[english]{babel}
\usepackage[all]{xy}
\theoremstyle{plain}

\def \demdu#1 { {\sl Proof #1.} }

\def \Gal {\text{\rm Gal}}

\def \tor {\text {\rm Tor}}

\def \limpro{\lim\limits_{\leftarrow} }
\def \limind{ \lim\limits_{\rightarrow} }
\def \qed{\text {$\quad \square$}}

\def \C {\overline {C}}
\def \U {\overline {U}}

\def \F#1 {\overline {F^\times_{#1}}}

\def\CM{{\mathbb{C}}}

\def\FM{{\mathbb{F}}}

\def\NM{{\mathbb{N}}}

\def\QM{{\mathbb{Q}}}

\def\ZM{{\mathbb{Z}}}

\def\al{\alpha}

\def\ga{\gamma}
\def\Ga{\Gamma}
\def\de{\delta}

\def\La{\Lambda}

\def\si{\sigma}

\def\ze{\zeta}

\def\Cti{{\widetilde{C}}}

\def\Wti{{\widetilde{W}}}

\def\Hha{{\widehat{H}}}

\def\uba{{\overline{u}}}

\def\Cba{{\overline{C}}}

\def\Uba{{\overline{U}}}

\def\Wba{{\overline{W}}}

\def\Gal{\mathop{\mathrm{Gal}}\nolimits}

\def\Im{\mathop{\mathrm{Im}}\nolimits}

\def\Ker{\mathop{\mathrm{Ker}}\nolimits}
\def\ker{\Ker}

\providecommand{\bysame}{\leavevmode\hbox to3em{\hrulefill}\thinspace}
\providecommand{\MR}{\relax\ifhmode\unskip\space\fi MR } 
\providecommand{\MRhref}[2]{}
\providecommand{\href}[2]{#2}
\def\Dbar{\leavevmode\lower.6ex\hbox to 0pt{\hskip-.23ex
    \accent"16\hss}D}
\def\cfac#1{\ifmmode\setbox7\hbox{$\accent"5E#1$}\else
    \setbox7\hbox{\accent"5E#1}\penalty 10000\relax\fi\raise 1\ht7
\hbox{\lower1.15ex\hbox to 1\wd7{\hss\accent"13\hss}}\penalty 10000
\hskip-1\wd7\penalty 10000\box7}
\def\cftil#1{\ifmmode\setbox7\hbox{$\accent"5E#1$}\else
    \setbox7\hbox{\accent"5E#1}\penalty 10000\relax\fi\raise 1\ht7
\hbox{\lower1.15ex\hbox to 1\wd7{\hss\accent"7E\hss}}\penalty 10000
\hskip-1\wd7\penalty 10000\box7}
 \def\Dbar{\leavevmode\lower.6ex\hbox to
0pt{\hskip-.23ex \accent"16\hss}D}
  \def\cfac#1{\ifmmode\setbox7\hbox{$\accent"5E#1$}\else
  \setbox7\hbox{\accent"5E#1}\penalty 10000\relax\fi\raise 1\ht7
  \hbox{\lower1.15ex\hbox to 1\wd7{\hss\accent"13\hss}}\penalty 10000
\hskip-1\wd7\penalty 10000\box7}
\def\cftil#1{\ifmmode\setbox7\hbox{$\accent"5E#1$}\else
  \setbox7\hbox{\accent"5E#1}\penalty 10000\relax\fi\raise 1\ht7
  \hbox{\lower1.15ex\hbox to 1\wd7{\hss\accent"7E\hss}}\penalty 10000
\hskip-1\wd7\penalty 10000\box7}

\def\Wind{\underset {\rightarrow \infty} {\Wti}}
\def\Cind{\underset {\rightarrow \infty} {\Cti}}
\providecommand{\href}[2]{#2}
\begin{document}
\title{Asymptotic cohomology of circular units}
\author{Jean-Robert Belliard}
\address{Laboratoire de math\'ematiques de Besan\c con UMR 6623,
Universit\'e de Franche-Comt\'e,
16 route de Gray,
25030 Besan\c con cedex,
FRANCE}
\email{jean-robert.belliard@univ-fcomte.fr}
\date{14th November 2007}
\frontmatter

\begin{abstract}
Let $F$ be a number field, abelian over $\QM$,  and fix a prime $p\neq 2$.
Consider the cyclotomic $\ZM_p$-extension $F_\infty/F$ and denote $F_n$ the ${n}^{\rm th}$ finite subfield and $C_n$ its group of circular units. Then the Galois groups $G_{m,n}=\Gal(F_m/F_n)$ act naturally on the $C_m$'s (for any $m\geq n>> 0$). We compute the Tate cohomology groups $\Hha^i(G_{m,n}, C_m)$ for $i=-1,0$ without assuming anything else neither on $F$ nor on $p$.
\end{abstract}

\thanks{This paper was written while I was hosted by the CNRS at the Institut de Math\'ematiques de Bordeaux. I thank both these institutions for their hospitality.}

\maketitle
\tableofcontents

\mainmatter
\section*{Introduction}
Let $F$ be a number field, abelian over $\QM$,  and fix a prime $p\neq 2$.
Consider the cyclotomic $\ZM_p$-extension $F_\infty/F$ and denote $F_n$ its ${n}^{\rm th}$ finite layer, i.e. $F_n/F$ is the unique sub-extension of degree $p^n$ of $F_\infty/F$. Following Sinnott (\cite{Si80}) we define, for every $F_n$, its group of circular units $C_n$. Then the cyclic Galois groups $G_{m,n}=\Gal(F_m/F_n)$ act naturally on the $C_m$'s (for any $m\geq n\geq 0$). The present note aims to compute the Tate cohomology groups $\Hha^i(G_{m,n}, C_m)$ for $i=-1,0$ without assuming anything else on $F$.  To complete our computations in this general case, we need to take into account differences between Sinnott and Washington versions of circular units, even at the infinite level. In the process, as it does not involve any other difficulties, we compute the cohomology of this second version, and also the $\Ga_n=\Gal(F_\infty/F_n)$-cohomology of both inductive limits.
Unfortunately we only succeed in describing these cohomology groups by assuming that $n$ is not less than some natural number $N$ (see lemma \ref{cond1}, lemma \ref{cond2} and theorem \ref{asy}), hence the word "asymptotic" in the title.
On the other hand if we do assume {\it ad hoc} hypotheses on $F$, then the natural number $N$ may be chosen to be equal to 0, and our general approach recovers all previously known special cases (see e.g. articles \cite{Kim92,Kim95,Kim99,KO01}, and \S III.4 of \cite{JNT1}, and \S 4 of \cite{NG06}).
\section{Two versions of cyclotomic units}
We start by recalling the definitions of Sinnott's and Washington's circular units. Fix an embedding of $\overline{\QM}$ in $\CM$ and denote $\ze_m=\exp(2i\pi/m)$.
\begin{defi}\label{circunits}
Let $F$ be an abelian number field and $f$ be its conductor, that is the smallest integer such that $F\subset \QM(\ze_f)$. The group of units of $F$ will be denoted $U(F)$.
\begin{enumerate}
\item We call group of circular numbers of $F$ and denote ${\mathrm Cyc}(F)$ the Galois submodule of $F^\times$ generated by $-1$ and the numbers
$N_{\QM(\ze_d) /\QM(\ze_d)\cap F} (1-\ze_d)$, where $d$ runs through all divisors of $f$ greater than 1.
\item We call Sinnott's group of circular units of $F$ (see \cite{Si80}) and denote $C(F)$ the intersection
$$C(F)= \mathrm{Cyc}(F) \bigcap U(F).$$
\item We call Washington's group of circular units of $F$ (see p.143 of \cite{Wa}) and denote $W(F)$ the intersection
$$W(F)= \mathrm{Cyc}(\QM(\ze_f)) \bigcap U(F).$$
\end{enumerate}
\end{defi}
Fix now a prime number $p\neq 2$ and consider the cyclotomic $\ZM_p$-extension $F_\infty/F$ of \underline{any abelian} number field $F$. Then all finite layers $F_n$ are still abelian over $\QM$ and the notations $C_n=C(F_n)$ and $W_n=W(F_n)$ make sense. We also abbreviate $U(F_n)$ to $U_n$. When necessary we will rather consider the pro-$p$-completions $\C(F)=C(F)\otimes \ZM_p$, $\Wba_n=W_n \otimes \ZM_p$, $\Cba_n=C_n\otimes \ZM_p$ and $\U_n=U_n\otimes \ZM_p$. We put as usual $\Ga=\Gal(F_\infty/F)$ and denote $\La=\ZM_p[[\Ga]]$ for the Iwasawa algebra.
For all $n$ we have obvious inclusions $C_n\subset W_n$. Counterexamples to the equality  are known. Moreover there exist also abelian number fields $F$ such that even the projective limits $\C_\infty =\limpro \C_n$ and $\Wba_\infty = \limpro \Wba_n$ with respect to norm maps along $F_\infty/F$ disagree (see e.g. \cite{pmb}; other examples were built simultaneously and independently by Ku{\v{c}}era in \cite{Kuc03}).
\begin{defi}\label{KN} Let us denote $KN_\infty$ the quotient module $KN_\infty=\Wba_\infty /\C_\infty$. This module has been defined as a capitulation kernel for unit classes and called "Ku{\v{c}}era--Nekov\' a\v r kernel" in \cite{BN2} and \cite{BNL}.
\end{defi}
Let $\U_\infty$ be the projective limit of the $\U_n$. Let $\overline{\mu}_n$ be the $p^{\mathrm{th}}$-power roots of unity in $F_n$ and  $\overline{\mu}_\infty$ be their projective limit. Let $F^+$ be the maximal real subfield of $F$, and for any $p$-adic $\Gal(F/F^+)$-module
$M$ let us abbreviate $M^{\Gal(F/F^+)}$ to $M^+$. We have $\overline{\mu}_\infty\subset \C_\infty\subset \Wba_\infty \subset \U_\infty$ and with $M$ standing for any of these four modules of units we have a direct sum decomposition  $M=M^+\oplus \overline{\mu}_\infty$.
For later use, we summarize some other known properties of these modules in a proposition.
\begin{prop}\label{Cinf}\ Let $r=[F^+:\QM]$.
\begin{enumerate}
\item The module $KN_\infty$ is finite.
\item The module $\Wba_\infty^+$ is $\La$-free of rank $r$.
\item The $\La$-module $\C_\infty^+$ is of rank $r$.
\item The module $\C_\infty^+$ is $\La$-free if and only if $KN_\infty=0$.
\item The module $KN_\infty$ is the maximal finite submodule of $\U_\infty/\C_\infty=\U_\infty^+/\C_\infty^+$.
\end{enumerate}
\end{prop}

\begin{proof} The indices $(W_n:C_n)$ are bounded uniformly in $n$, this is the main result of \cite{KN95}. Assertion 1 follows.  Assertion 2 is a theorem of Kuz$'$min (see \cite{Ku96} or discussion before proposition 3.6 of \cite{pmb}). Assertion 3 is a consequence of 1 and 2. Assertion 4 is proposition 3.6 of \cite{pmb} (see also \cite{BN2}, proposition 2.3 $(ii)$). Let us denote $MF_\infty$ for the maximal finite submodule of $\U_\infty/\C_\infty$. By assertion 1 we already have an inclusion
$KN_\infty \subset MF_\infty$. But since both $\U_\infty^+$ and $\Wba_\infty^+$ are $\La$-free the quotient $\U_\infty /\Wba_\infty$ has no non-trivial finite submodule. This shows assertion 5.\end{proof}
In \cite{BN2} (proposition 2.3 $(ii)$) the module denoted $KN_F$ was shown to be equal to $MF_\infty$, so that notations are coherent. For our present purpose we will make no use of the interpretation in terms of capitulation kernels. In the sequel we will consider this finite module $KN_\infty$ as a parameter depending on the base field $F$.
As written above, two different infinite families of couples $(F,p)$ giving non-trivial $KN_\infty$ are described in
\cite{Kuc03} and \cite{pmb}. On the other hand the definition 6.4, together with theorem 6.6 in the appendix to \cite{BNL}  gives a criterion for the triviality of $KN_\infty$. Up to now all known examples of trivial $KN_\infty$ satisfy this criterion.

\section{Universal norms of circular units}

To compute the cohomology of $C_m$ and $W_m$ we first use the simpler structure of $\Wba_\infty^+$ and of $\C_\infty^+$ to recover the cohomology of the "universal norms" modules
$\Cti_n^+\subset \C_n$ and $\Wti_n^+\subset \Wba_n$, which are defined as the images of the usual projection maps. Then in section \ref{coh} we will control the deviation between the initial modules and their universal norms.
\begin{defi} Let $n\in \NM$.
\begin{enumerate}
\item Let $\Cti_n=\Im(\C_\infty \longrightarrow \C_n)$ be the module consisting of the universal norms of Sinnott's units.
\item Let $\Wti_n=\Im(\Wba_\infty \longrightarrow \Wba_n)$ be the module consisting of the universal norms of Washington's units.
\end{enumerate}
\end{defi}
By a usual compactness argument we have e.g. $\Cti_n=\bigcap_{m\geq n} N_{F_m/F_n} (\Cba_m)$ hence the terminology "universal norms".

In the sequel we will obtain asymptotic results but will try to collect as much information as possible about the first layer $F_n$ from which our results apply. For the cohomology of $\Wti_n$, this is easier. It is well known that there exists some integer $n$ such that $F_\infty/F_n$ is totally ramified at every place above $p$.
\begin{defi}\label{nd} We will denote $n_d$ the smallest integer such that no place above $p$ do splits anymore in $F_\infty/F_{n_d}$.\end{defi}
If $p$ is (at most) tamely ramified in $F/\QM$, then $n_d=0$.
For all $m\geq n \geq 0$ we will denote $\Gal(F_\infty/F_n)$ by $\Ga_n$ and $\Gal(F_m/F_n)\cong \Ga_n/\Ga_m$ by $G_{m,n}$.
By the mere definition of Galois action on places, the group $\Ga_{n_d}$ acts trivially on the set $S_\infty$ of places of $F_\infty$ dividing $p$.

For later use we will isolate in a lemma here a precise statement for our context of the well known snake lemma :
\begin{lemm}\label{Sn}
For all $m\in \NM$ and all $\La$-module $M$ let us denote by $M^{\Ga_m}$ the $\Ga_m$-invariant submodule of $M$ and by $M_{\Ga_m}$ the $\Ga_m$-coinvariant quotient module of $M$. Let $0\rightarrow A\rightarrow B\overset \varphi \rightarrow C\rightarrow 0$ be an exact sequence of $\La$-module. Pick for all $m$ a topological generator $\ga_m$ of $\Ga_m$. Then for all $m\in \NM$ the map $\de\colon C^{\Ga_m}\longrightarrow A_{\Ga_m}$ is well defined by the formula $\de(\varphi(b))=(\ga_m-1) b + (\ga_m-1) A$ and fits into the exact sequence
$$0\longrightarrow A^{\Ga_m} \longrightarrow B^{\Ga_m} \longrightarrow C^{\Ga_m} \overset \de \longrightarrow
A_{\Ga_m} \longrightarrow B_{\Ga_m} \longrightarrow C_{\Ga_m}\longrightarrow 0.$$
\end{lemm}
\begin{proof} Just apply snake lemma to the sequence $0\rightarrow A\rightarrow B\overset \varphi \rightarrow C\rightarrow 0$ connected with itself by multiplication by $\ga_m-1$.
\end{proof}
\begin{prop}\label{CoWti} Let $s^+=\# S_\infty^+$ be the number of places of $F^+_\infty$ dividing $p$. Let $n_d$ be as in definition \ref{nd}. Then for all $m\geq n\geq n_d$, we have
$\Hha^0(G_{m,n},\Wti_m)=0$ and $H^1(G_{m,n},\Wti_m)$ is free of rank $s^+$ over $\ZM/p^{m-n}\ZM$.
\end{prop}
\begin{proof} Fix $m\geq n \geq n_d$. The canonical surjection $\Wba_\infty \longrightarrow \Wti_m$ factors through $w_m\colon (\Wba_\infty)_{\Ga_m} \twoheadrightarrow \Wti_m$. Set $T_m=\ker w_m$. By proposition \ref{Cinf} we know that $(\Wba_\infty)_{\Ga_m}$ is isomorphic to $\overline{\mu}_m \oplus (\Wba^+_\infty)_{\Ga_m}$ with both summands $G_{m,0}$-cohomologically trivial. Therefore the cohomology of
$\Wti_m$ is the cohomology of $T_m$ shifted once. To conclude the proof of proposition \ref{CoWti} it suffices to show that $T_m$ is a free $\ZM_p$-module of rank $s^+$ with trivial action of $G_{m,n}$. For all $n$ let $\U'_n$ denote the pro-$p$-completion of $(p)$-units of $F_n$, and $\U'_\infty$ be their projective limit. By two theorem of Kuz$'$min (see theorems 7.2 and 7.3 in \cite{Ku72}) the module $\U'_\infty$ is $\La$-free and for all $m$ the map $(\U'_\infty)_{\Ga_m} \longrightarrow \U'_m$ is injective.

The sequence
$\xymatrix{0\ar[r] & \Wba_\infty \ar[r] & \U'_\infty \ar[r] &  \U'_\infty/\Wba_\infty\ar[r] &0 }$ gives by lemma \ref{Sn}
$\xymatrix{0\ar[r] & (\U'_\infty/\Wba_\infty)^{\Ga_m} \ar[r] & (\Wba_\infty)_{\Ga_m} \ar[r] & (\U'_\infty)_{\Ga_m}}$ hence an isomorphism $$T_m\cong (\U'_\infty/\Wba_\infty)^{\Ga_m}.$$
This, together with the following lemma \ref{Tm}, completes the proof of the proposition \ref{CoWti}.\end{proof}
\begin{lemm}\label{Tm} Let $s_m^+$ be the cardinality of the set $S_m^+$ of places dividing $p$ in the maximal real subfield $F_m^+$ of $F_m$. Note that for all $m\geq n_d$ the set $S_m^+$ is in bijection with $S_\infty^+$.
\begin{enumerate}
\item The module $\U'_\infty/\U_\infty$ is $\ZM_p$-free and pseudo-isomorphic to $\ZM_p[S_\infty^+]$.
\item For all $m$ the module $(\U'_\infty/\Wba_\infty)^{\Ga_m}$ is $\ZM_p$-free of rank $s^+_m$ with trivial action by $\Ga_{n_d}$.
\end{enumerate}
\end{lemm}
\begin{proof}
We first prove that $1$ implies $2$, then we will prove $1$.
Apply lemma \ref{Sn} to the sequence
$$\xymatrix{0\ar[r] & \U_\infty/\Wba_\infty \ar[r] & \U'_\infty/\Wba_\infty \ar[r] & \U'_\infty/\U_\infty \ar[r] & 0.}$$ We get $$\xymatrix{0\ar[r] & (\U_\infty/\Wba_\infty)^{\Ga_m} \ar[r] & (\U'_\infty/\Wba_\infty)^{\Ga_m} \ar[r] & (\U'_\infty/\U_\infty)^{\Ga_m} \ar[r] & (\U_\infty/\Wba_\infty)_{\Ga_m}}$$
By lemma \ref{Cinf} the module $\U_\infty/\Wba_\infty$ is pseudo-isomorphic to $\U_\infty/\C_\infty$, and as a consequence of Coleman's theory and Leopoldt's conjecture (see theorem 1.1 of \cite{CJM}) these two modules have finite $\Ga_m$-invariants and co-invariants. But since $\U_\infty^+$ and $\Wba_\infty^+$ are $\La$-free the quotient $\U_\infty/\Wba_\infty$ has no non-trivial finite submodule so that $(\U_\infty/\Wba_\infty)^{\Ga_m}=0$ and our previous sequence becomes $$\xymatrix{0\ar[r]  & (\U'_\infty/\Wba_\infty)^{\Ga_m} \ar[r] & (\U'_\infty/\U_\infty)^{\Ga_m} \ar[r] & {\mathrm{finite}}.}$$
Therefore assertion 2 follows from assertion 1.
To prove $1$ we use (normalized) valuations at places above $p$ and consider the exact sequence~:
$$\xymatrix{0\ar[r] & \U'_\infty/\U_\infty \ar[r]^-{\mathrm{val}}&\ZM_p[S_\infty]\ar[r] & D_\infty\ar[r] & 0,}$$
where $D_\infty$ stands for (the projective limit) of the $p$-part of the subgroup of ideal class groups generated by places above $p$. This shows already that $(\U'_\infty/\U_\infty)$ is $\ZM_p$-free with trivial action of $\Ga_{n_d}$.
Proving that $(\U'_\infty/\U_\infty)$ is pseudo-isomorphic to $\ZM_p[S_\infty^+]$ is then equivalent to prove that $D_\infty$ is pseudo-isomorphic to $\ZM_p[S_\infty]^-=\ZM_p[S_\infty]/\ZM_p[S_\infty^+]$. This last statement is a consequence of Leopoldt conjecture (which holds true for our abelian field $F$) and is part of the folklore. For instance Greenberg in \cite{Gree73} \S1 proves that $D_\infty^+$ is finite and constructs in loc. cit. \S 2 a sequence of subgroups (denoted $C_m$ in loc. cit. p.120) of uniformly bounded finite index in $D_m$, whose projective limit is pseudo-isomorphic to $\ZM_p[S_\infty]^-$.
\end{proof}
\begin{coro} Let $\Wind={\limind}_{m} \Wti_m$ be the inductive limit of the $\Wti_m$'s. For all $n\geq n_d$ we have a (group)-isomorphism $H^1(\Ga_n,\Wind)\simeq (\QM_p/\ZM_p)^{s^+}$ and $H^2(\Ga_n,\Wind)$ is trivial.
\end{coro}
\begin{proof} The groups $H^i(\Ga_n,\Wind)$ are isomorphic to the inductive limit with respect to the maps
$H^i(G_{m,n},\Wti_m)\longrightarrow H^i(G_{m+1,n},\Wti_{m+1})$
induced by the couples of maps $G_{m+1,n} \rightarrow G_{m,n}$ and $\Wti_m\rightarrow \Wti_{m+1}$ (see proposition 1.5.1 of \cite{NSW}). By proposition \ref{CoWti} the $\Wti_m$ satisfies Galois descent, therefore the inflation maps $H^1(G_{m,n},\Wti_m)\longrightarrow H^1(G_{m+1,n},\Wti_{m+1})$ are injectives. As $H^1(G_{m,n},\Wti_m)\simeq (\ZM/p^{m-n})^{s^+}$ the first limit is $(\QM_p/\ZM_p)^{s^+}$. Triviality of all $H^2$ follows from that of all the $\Hha^0$'s.
\end{proof}

We want to describe $G_{m,n}$-cohomology of the $\Cti_m$ and we will use a variation on the above method used for the $\Wti_m$. As at the beginning of the proof of proposition \ref{CoWti} we have an isomorphism
$\ker((\C_\infty)_{\Ga_m} \longrightarrow \Cti_m)\cong (\U'_\infty/\C_\infty)^{\Ga_m}$. But $\U'_\infty/\C_\infty$ may have a non trivial finite submodule. Indeed we have sequence
\begin{eqnarray}\label{Vms}\xymatrix{0\ar[r] & KN_\infty^{\Ga_m} \ar[r] &  (\U'_\infty/\C_\infty)^{\Ga_m} \ar[r] & (\U'_\infty/\Wba_\infty)^{\Ga_m}\ar[r] & {\mathrm{finite}} ,}
\end{eqnarray} extracted from the lemma \ref{Sn} applied to
the sequence $0\rightarrow KN_\infty \rightarrow \U'_\infty/\C_\infty \rightarrow \U'_\infty/\Wba_\infty \rightarrow 0$.
\begin{defi}\label{Vm}
We define  $V_m:=(\U'_\infty/\C_\infty)^{\Ga_m}/KN_\infty^{\Ga_m}$.
\end{defi}
From the sequence (\ref{Vms}) and lemma \ref{Tm} we see that
$V_m$ is $\ZM_p$-free of rank $s^+$ and the natural action by $\Ga_{n_d}$ on $V_m$ is trivial. Hence for $n\geq n_d$ the $G_{m,n}$-cohomology of $V_m$ is just $s^+$ copies of the cohomology of $\ZM$. We will get the cohomology of $\Cti_m$ by going through the cohomology of the sequence~:
\begin{eqnarray}\label{resCti} \xymatrix{0\ar[r] & V_m \ar[r] & (\C_\infty)_{\Ga_m} / KN_\infty^{\Ga_m} \ar[r] & \Cti_m \ar[r] & 0}
\end{eqnarray}
Before this computation  we will prove the existence of a first layer $n$ from which our result will apply and try to make it as precise as possible.
\begin{lemm}\label{cond1} \

\begin{enumerate}
\item There exists an $n$ such that for all integers $m\geq n$ we have $$\left (\Uba'_\infty/\Cba_\infty\right )^{\Ga_{m}} =
\left (\Uba'_\infty/\Cba_\infty\right )^{\Ga_n}.$$
\item For any integer $n$ satisfying 1, the group $\Ga_n$ acts trivially on $KN_\infty$.
\item If $n$ satisfies 1, then $n \geq n_d$.
\item If $KN_\infty=0$ then the integer $n_d$ of definition \ref{nd} satisfies 1.
\end{enumerate}
\end{lemm}
\begin{proof}
The $\La$-module $\U'_\infty/\C_\infty$ is finitely generated. The  sequence of submodules $\left ((\Uba'_\infty/\Cba_\infty)^{\Ga_{m}}\right)_{m\in\NM}$ is increasing. Part 1 of the lemma follows by noetherianity.
Part 2 comes from the inclusion $KN_\infty \subset \U'_\infty/\C_\infty$.
Suppose that some $x\in KN_\infty$ is such that $\si x\neq x$ for some $\si \in \Ga_n$. But $x$, as an element of the finite module $KN_\infty$, must be fixed by some $\Ga_m$ for $m> n$. Hence $x\in \left (\Uba'_\infty/\Cba_\infty\right )^{\Ga_{m}} \setminus
\left (\Uba'_\infty/\Cba_\infty\right )^{\Ga_n}$ and $n$ does not satisfy 1. This shows assertion 2. For assertion 3, and 4 we have seen in the proof of lemma \ref{Tm} that $(\U'_\infty/\Wba_\infty)^{\Ga_m}$ embeds with finite cokernel into $\ZM_p[S_\infty]$ for all $m\geq n_d$. Now $\Ga_{n_d}$ acts trivially on $\ZM_p[S_\infty]$ and if $n_d>0$ then
$\ZM_p[S_\infty]^{\Ga_{n_d-1}}$ is not of finite index in $\ZM_p[S_\infty]$. This proves 3 and 4.
\end{proof}
\begin{defi}\label{nUC} We will denote $n_{U'C}$ the smallest integer $n$ such that for all integers $m\geq n$ the equality $\left (\Uba'_\infty/\Cba_\infty\right )^{\Ga_{m}} =
\left (\Uba'_\infty/\Cba_\infty\right )^{\Ga_n}$ holds.
\end{defi}
This number $n_{U'C}$ is the first nonexplicit part of our "asymptotic" approach. Other approaches found in the literature simply assume hypotheses that imply $KN_\infty=0$ and in such case we simply have $n_{U'C}=n_d$.

\begin{theo}\label{CoCti} Let $s^+$ be as in proposition \ref{CoWti}. Let $n$ be greater than or equal to $n_{U'C}$.
Let $V_m$ be the $\ZM_p$-free module of rank $s^+$ and trivial $\Ga_{n_d}$-action of definition \ref{Vm}.
For every integer $i$ let $KN_\infty[p^i]$ be the kernel of multiplication by $p^i$ in $KN_\infty$.
For every $m\geq n$ we have
$$\Hha^0(G_{m,n},\Cti_m)\cong KN_\infty[p^{m-n}]$$ and a (group split) exact sequence
$$\xymatrix{0\ar[r] &KN_\infty/p^{m-n} \ar[r] &H^{1}(G_{m,n},\Cti_m)\ar[r] &V_m/{p^{m-n}}\ar[r] & 0}.$$
\end{theo}
\begin{proof} Fix $m\geq n\geq n_{U'C}$. To describe the cohomology of the sequence (\ref{resCti}), we first compute the cohomology of $(\C_\infty)_{\Ga_m} / KN_\infty$. From the sequence $0\rightarrow \C_\infty\rightarrow \Wba_\infty \rightarrow KN_\infty\rightarrow 0$ we obtain by lemma \ref{Sn} a sequence
$$\xymatrix{0\ar[r] & (\C_\infty)_{\Ga_m} / KN_\infty \ar[r] & (\Wba_\infty)_{\Ga_m} \ar[r] & KN_\infty \ar[r] & 0 .}$$ As $(\Wba_\infty)_{\Ga_m}$ is $G_{m,n}$-cohomologicaly trivial we see that $(\C_\infty)_{\Ga_m} / KN_\infty$ has the $G_{m,n}$-cohomology of $KN_\infty$ shifted once. The exact hexagone of cyclic Tate cohomology associated to the sequence $(\ref{resCti})$ is now~:
$$\xymatrix { V_m/p^{m-n} \ar[r]^-\al & KN_\infty[p^{m-n}] \ar[r] & \Hha^0(G_{m,n},\Cti_m) \ar[d] \\
\Hha^{-1}(G_{m,n},\Cti_m) \ar[u] & KN_\infty/p^{m-n} \ar[l] & 0 \ar[l] .}$$
To conclude the proof of the theorem it remains now to prove that the map $\al$ of this diagram is the $0$ map.
This map $\al$ is induced by the map $\de\colon (\U'_\infty/\C_\infty)^{\Ga_m} \longrightarrow (\C_\infty)_{\Ga_m}$ of the lemma \ref{Sn} applied to the sequence $0\rightarrow \C_\infty\rightarrow \U'_\infty \rightarrow \U'_\infty/\C_\infty \rightarrow 0$.
To define this map $\de$ let us pick a choice of a generator $\ga_m$ of $\Ga_m$, and for later use let us denote
$\ga_n$ the generator of $\Ga_n$ such that $\ga_n^{p^{m-n}}=\ga_m$.
Then the map $\de$ sends any coset $\uba \in \U'_\infty/\C_\infty$ such that $(\ga_m-1) u \in \C_\infty$ to the coset of $(\ga_m-1) u$ in $(\C_\infty)_{\Ga_m}$. Note that $\ga_m-1= \nu_{m,n} (\ga_n-1)$, where $\nu_{m,n}=\sum_{i=0}^{p^{m-n}-1} \ga_n^i\in \La$ acts like the algebraic norm of the extension $F_m/F_n$ on any $\Ga_m$-trivial module. But the integers $m\geq n$ have been chosen in such a way that $(\U'_\infty/\C_\infty)^{\Ga_m}=(\U'_\infty/\C_\infty)^{\Ga_n}$ and therefore we have $(\ga_m-1) u = \nu_{m,n} (\ga_n-1) u$, where $(\ga_n-1) u \in \C_\infty$. As by definition we have $$\Hha^0(G_{m,n}, (\C_\infty)_{\Ga_m}/KN_\infty) = \frac {((\C_\infty)_{\Ga_m}/KN_\infty)^{G_{m,n}}} {\nu_{m,n} ((\C_\infty)_{\Ga_m}/KN_\infty)},$$ this shows that $\al$ is indeed the $0$ map and concludes the proof of theorem \ref{CoCti}.
\end{proof}
\begin{coro} Let $\Cind={\limind}_{m} \Cti_m$ be the inductive limit of the $\Cti_m$'s. For all $n\geq n_{U'C}$
we have (group)-isomorphism $H^1(\Ga_n,\Cind)\simeq (\QM_p/\ZM_p)^{s^+}$ and $H^2(\Ga_n,\Cind)$ is trivial.
\end{coro}
\begin{proof}
As $\Ga_n$ is pro-$p$-free $H^2(\Ga_n,\Cind)$ is divisible and its order is bounded by the finite order of $KN_\infty$, hence $H^2(\Ga_n,\Cind)$ is trivial. Let us compute the $H^1(\Ga_n,\Cind)$. Extracted mutatis-mutandis from the proof of the theorem \ref{CoCti} we have the exact sequence for all $m>>0$~:
$$0\longrightarrow H^1(G_{m,n}, (\C_\infty)_{\Ga_m}/KN_\infty) \longrightarrow H^1(G_{m,n},\Cti_m) \longrightarrow H^2(G_{m,n}, V_m)\longrightarrow 0.$$
Of course we have for $m>>0$ the isomorphisms $H^1(G_{m,n}, (\C_\infty)_{\Ga_m}/KN_\infty)\simeq KN_\infty$ and
$H^2(G_{m,n}, V_m)\simeq V_m/p^{m-n}$ and we recover the exact sequence~:
$$0\longrightarrow KN_\infty \longrightarrow H^1(G_{m,n},\Cti_m) \longrightarrow V_m/p^{m-n}\longrightarrow 0.$$
We want to take inductive limit on this sequence, and for that we need to clarify what are the extension maps going up from the $m$th to the $(m+1)$th step for the modules $KN_\infty$ and the modules $V_m/p^{m-n}$.
The finite module of universal norms $KN_\infty$ stabilizes with respect to natural (going down) norm maps and composition of the two maps give multiplication by $p$. Therefore the inductive limit with respect to extension maps of $KN_\infty$ is trivial. On the contrary the modules $(V_m)_m$ stabilize with respect to extension maps already from the first step to  $V_n$ for $m\geq n\geq n_{U'C}$. So for fixed $n\geq n_{U'C}$ the inductive limit $\limind V_m/p^{m-n}$ is $V_n\otimes \QM_p/\ZM_p \simeq (\QM_p/\ZM_p)^{s^+}$.
\end{proof}
\section{Cohomology of circular units}\label{coh}
We will try and recover the cohomology of $\C_m$ and $\Wba_m$ from that of their universal norms $\Cti_m$ and $\Wti_m$. For $\C_m$ the method is not new and has already been used in \cite{NG06} or \cite{BNL} and originally in \cite{JNT1}. Here we only encounter another slight difficulty coming from the non triviality of $\Hha^0(G_{m,n},\Cti_m)$. For $\Wba_m$ our result is original.
\begin{lemm}\label{cond2} For all $n$ let $I_n$ be the inertia subfield for $p$ in $F_n/\QM$. We have $I_{n+1}=I_n$ as soon as $F_{n+1}/F_n$ ramifies. Let $I=\bigcup_n I_n$ be the inertia subfield for $p$ in $F_\infty/\QM$.
\begin{enumerate}
\item For all $n$ we have equality
$$\C_n=\Cti_n \C(I_n),$$
hence for all $n$ such that $F_{n+1}/F_n$ ramifies we have equality
$$\C_n=\Cti_n \C(I).$$
\item\label{itcond2} There exist an integer $n$ such that for all integers $m\geq n$ we have
$$\Cti_{m} \bigcap C(I) = \Cti_{n} \bigcap C(I).$$
\end{enumerate}
\end{lemm}
\begin{proof} Assertion 1 is lemme 2.5 of \cite{pmb}.
The $\ZM_p$-module $C(I)$ is finitely generated. The sequence of submodules  $\left (\Cti_m\bigcap C(I)\right )_{m\in\NM}$ is increasing. Assertion 2 follows by noetherianity.
\end{proof}
\begin{defi} We will denote $n_C$ the smallest non negative integer such that $F_{n_C+1}/F_n$ ramifies and such that property \ref{itcond2} of lemma \ref{cond2} holds.
\end{defi}
This number $n_C$ is our second nonexplicit asymptotic constant.
However if we assume $KN_\infty=0$ then theorem \ref{CoCti} shows that the $\Cti_m$'s satisfy Galois descent and that proves the equality $n_C=n_i$, where $n_i$ is the smallest non negative integer such that $F_{n_i+1}/F_{n_i}$ ramifies.

After $KN_\infty$ there is another module that we have to consider as a parameter depending on $F$. This module is essentially the one denoted $\Phi$ in \cite{NG06} and that we will define in our more general context now~:
\begin{defi} We call universal co-norms of circular units and  denote by $\Phi_m$ the quotient module $$\Phi_m=\C_m/\Cti_m.$$
\end{defi}
We state and prove now the properties of $\Phi_m$ which are needed to compute the cohomology of $\C_m$. In particular we fix its asymptotic behavior, so that we may
consider it as a parameter (even if it is an non explicit and asymptotic one !). We will give somewhat more explicit informations about this module and examine some examples in section \ref{UCN}.
\begin{lemm}\label{Philem} Recall that $n_i$ is the first nonnegative integer such that $F_{n_i+1}/F_{n_i}$ ramifies.
\begin{enumerate}
\item For all $m$ the $\ZM_p$-rank of $\Phi_m$ is $s^+_m-1$ where $s^+_m$ is the cardinal of the set $S^+_m$ of places dividing $p$ in the maximal real subfield $F_m^+$ of $F_m$.
\item For all $m$ the group $\Ga_{n_i}$ acts trivially on $\Phi_m$.
\item For all $m\geq n_{C}$ the extension map $\C_{n_{C}} \longrightarrow \C_m$ induces an isomorphism
$$\Phi_{n_{C}} \cong \Phi_m.$$
\end{enumerate}
\end{lemm}
\begin{proof} This lemma is an easy generalization of lemma-definition 2.2 in \cite{BNL}, but for the convenience of the reader we reprove it. To compute the rank of $\Phi_m$ just use the exact sequence
$$\xymatrix {0\ar[r] & (\U'_\infty/\C_\infty)^{\Ga_m}  \ar[r] & (\C_\infty)_{\Ga_m}\ar[r] & \C_m \ar[r] & \Phi_m \ar[r] & 0}.$$
Let $r=[F^+:\QM]$. Then the $\La$-module $\C_\infty$ has trivial $\Ga$-invariants and rank $r$. It follows that
$(\C_\infty)_{\Ga_m}$ has $\ZM_p$-rank equal to $rp^m$. Also the $\ZM_p$-module $\C_m$ is of finite index in $\U_m$ and has rank $rp^m-1$ by Dirichlet's theorem. The module $(\U'_\infty/\C_\infty)^{\Ga_m}$ has the same rank as $(\U'_\infty/\Wba_\infty)^{\Ga_m}$ i.e. $s_m^+$ by assertion 2 of lemma \ref{Tm}. This proves 1. Assertions 2 and 3 follow from lemma \ref{cond2}. Indeed for all $m$ we have
$$\Phi_m= \C_m/\Cti_m = \Cti_m C(I_m)/ \Cti_m \cong C(I_m)/\Cti_m \bigcap C(I_m).$$
For all $m$ the action of $\Ga_{n_i}$ is trivial on $I_m$, which proves 2 and for all $m\geq n_{C}$ the extension map $\Phi_{n_C} \longrightarrow \Phi_m$ commutes with
the identity map in $$\C(I)/(\Cti_{n_C}\cap \C(I)) = \C(I)/(\Cti_m\cap \C(I)),$$ which proves 3.
\end{proof}
In the sequel we will abbreviate $\Phi_{n_C}$ to simply $\Phi$.
Recall that from theorem \ref{CoCti} we have for all $m\geq n\geq n_{U'C}$ isomorphisms $\Hha^{0}(G_{m,n},\Cti_m)\cong KN_\infty[p^{m-n}]$ and exact sequences~:
$$\xymatrix{0\ar[r] &KN_\infty/p^{m-n} KN_\infty \ar[r] &\Hha^{-1}(G_{m,n},\Cti_m)\ar[r] & (\ZM/{p^{m-n}})^{s^+} \ar[r] & 0}.$$
\begin{theo}\label{asy} Let $n$ be any integer not less than both $n_{U'C}$ and $n_C$.
Then for all $m\geq n$ we have exact sequences
$$\xymatrix{ 0\ar[r] & KN_\infty[p^{m-n}]  \ar[r] & \Hha^0(G_{m,n}, \Cba_m) \ar[r] & \Phi/p^{m-n} \Phi \ar[r] &  0}$$
and
$$\xymatrix{ 0\ar[r] & H^{1}(G_{m,n},\Cti_m) \ar[r] & H^{1}(G_{m,n}, \Cba_m) \ar[r] & \Phi[p^{m-n}] \ar[r] &  0}.$$
\end{theo}
\begin{proof} Fix $m\geq n$ with $n\geq \max(n_C,n_{U'C})$. As $G_{m,n}$ is cyclic we may and will compute $\Hha^{-1}$ during the proof. Then this proof consists in splitting into two pieces the exact hexagone of cyclic Tate cohomology groups associated to the exact sequence $0\longrightarrow \Cti_m\longrightarrow \C_m \longrightarrow \Phi \longrightarrow 0$. At first sight this exact hexagone is
$$\xymatrix{ KN_\infty[p^{m-n}]  \ar[r]^-{\varphi_{0}} & \Hha^0(G_{m,n}, \Cba_m) \ar[r] & \Phi/p^{m-n} \Phi  \ar[d]^-{\de_0} \\
\Phi[p^{m-n}]\ar[u]^-{\de_{-1}} & \Hha^{-1}(G_{m,n}, \Cba_m) \ar[l] & \Hha^{-1}(G_{m,n},\Cti_m) \ar[l]^-{\varphi_{-1}} .}$$
So we only have to prove that the two $\de$ maps are both the $0$ map or equivalently that the $\varphi$ maps are injective. The injectivity of $\varphi_{-1}$ follows from the stabilization with respect to extension maps of the quotients $\Phi_m$. Indeed from this stabilization property we can write $\C_m=\C_n \Cti_m$ and therefore for all generator $\si$ of $G_{m,n}$ we get $(\si-1) (\C_m)=(\si-1) (\Cti_m)$ which proves that $\varphi_{-1}$ is injective.

The kernel of $\varphi_0$ is by definition the quotient $N_{m,n}(\C_m)\cap\Cti_m/N_{m,n} (\Cti_m)$, where $N_{m,n}$ stands for the norm map of $F_m/F_n$. But by assertion 2 of lemma \ref{cond2} we have
$$N_{m,n}(\C_m)=N_{m,n}(\Cti_m \C(I))=N_{m,n}(\Cti_m) \C(I)^{p^{m-n}}.$$ Hence the equalities
$$\ker \varphi_0=\frac {\left (N_{m,n}(\Cti_m) \C(I)^{p^{m-n}} \right )\bigcap \Cti_m } {N_{m,n}(\Cti_m)}=
\frac {N_{m,n}(\Cti_m) \left (\C(I)^{p^{m-n}} \bigcap \Cti_m\right ) } {N_{m,n}(\Cti_m)}.$$
But by the definition of $n_C$ for $m\geq n\geq n_C$ we have $\C(I)\cap\Cti_m =\C(I)\cap \Cti_n$ and by definition of $\Cti_n$ we have $\Cti_n\subset N_{m,n} (\Cti_m)$. This proves that $\ker \varphi_0=0$.
\end{proof}
\begin{coro}\label{infC} Let $\overset \rightarrow {C}_\infty$ be the inductive limit with respect to extension maps of the $\C_m$'s. Let $\tor(\Phi)$ be the $\ZM_p$-torsion of $\Phi$. For all $n$ not less than both $n_C$ and $n_{U'C}$ we have (group) isomorphisms~:
$$H^1(\Ga_n,\overset \rightarrow {C}_\infty)\simeq (\QM_p/\ZM_p)^{s^+}\oplus \tor(\Phi), $$
and
$$H^2(\Ga_n,\overset \rightarrow {C}_\infty)\simeq (\QM_p/\ZM_p)^{s^+-1}.$$
\end{coro}
\begin{proof} As before, when taking inductive limits, all contributions from the module $KN_\infty$ vanish. The module $\Phi$, contrary to $KN_\infty$, stabilizes with respect to extension maps. Therefore extensions for $\Phi[p^k]\longrightarrow \Phi[p^{k+1}]$ ultimately are isomorphisms and give $\tor(\Phi)$ as inductive limit, while extensions for $\Phi/p^{k} \longrightarrow \Phi/p^{k+1}$ are right from the start induced by multiplication by $p$ and give $\Phi \otimes (\QM_p/\ZM_p)\simeq (\QM_p/\ZM_p)^{s^+-1}$ as inductive limit. The sequence with $H^1$ is group split because $(\QM_p/\ZM_p)^{s^+}$ is an injective group.
\end{proof}
Finally we want to describe the cohomology of the $W_m$'s. We use the same strategy as for the $C_m$'s and for that, we first have to prove the stabilization by extension maps of the $\Wba_m/\Wti_m$'s. We will deduce this stabilization from the one of the $\Phi_m$'s, which unfortunately gives us a even worst lower bound $n_W$ from which our results apply.
\begin{lemm}\label{Wasy} Exists an $n$ such that for all $m\geq n$ the extension map $\Wba_n\longrightarrow\Wba_m$ induces an isomorphism $\Wba_n/\Wti_n \cong \Wba_m/\Wti_m$.
\end{lemm} \begin{proof} First take  an $n\geq n_C$, and consider the (tautological) diagram with exact rows~:
$$\xymatrix@=12pt{0\ar[r]& (\Wti_n \cap \C_n)/\Cti_n \ar[r]& \C_n/\Cti_n \ar[rr]\ar@{->>}[rd]& & \Wba_n/\Wti_n \ar[r]& \Wba_n/(\C_n\Wti_n) \ar[r]& 0 \\ & (\dag)& & \C_n/(\Wti_n\cap \C_n)\ar@{^(->}[ur] & & & }$$
As $n\geq n_C$ the extension map $\Phi_n\longrightarrow \Phi_m$ is  surjective and so is the extension map $\C_n/(\Wti_n\cap \C_n)\longrightarrow \C_m/(\Wti_n\cap \C_m)$. Recall that we have $n\geq n_C\geq n_i\geq n_d$, so that theorem \ref{CoWti} applies. In particular the $\Wti_m$'s satisfy Galois descent for $m\geq n$ and the extension
map $\Wba_n/\Wti_n \longrightarrow \Wba_m/\Wti_m$ is injective. Now use the snake lemma on the right hand part of the diagram $(\dag)$ with extension maps. This gives :
$$\xymatrix{ &  & 0\ar[d]\ar[r] & K_{n,m} \ar[d] & \\
0\ar[r] & \C_n/(\Wti_n\cap \C_n) \ar[r]\ar[d] & \Wba_n/\Wti_n \ar[r]\ar[d]& \Wba_n/(\C_n\Wti_n) \ar[r]\ar[d]& 0 \\
0\ar[r] & \C_m/(\Wti_m\cap \C_m) \ar[r]\ar[d] & \Wba_m/\Wti_m \ar[r]\ar[d]& \Wba_m/(\C_m\Wti_m) \ar[r]\ar[d] & 0 \\
& 0 \ar[r]&CoK_{n,m}\ar[r] & CoK'_{n,m}\ar[r]&0.}$$
It follows that the kernels $K_{n,m}$ are trivial and that the co-kernels $CoK_{n,m}$ and $CoK'_{n,m}$ are isomorphic. By the triviality of the kernels we see that the orders of the finite groups $\Wba_n/\C_n\Wti_n$ are increasing with $n$, but must stabilize because they are uniformly bounded by the maximal of the orders of $\C_n/\Wba_n$ (which exists according to main result of \cite{KN95}). So for an $n$ greater than $n_C$ and such that the order of $\Wba_n/\C_n\Wti_n$ is maximal, we get the triviality of the two cokernels $CoK_{n,m}$.\end{proof}
\begin{defi} We will denote by $n_W$ the smallest integer not less than $n_d$ and such that for all $m\geq n$ the extension map $\Wba_n/\Wti_n\longrightarrow \Wba_m/\Wti_m$ is an isomorphism. We will abbreviate $\Wba_m/\Wti_m$ to $\Phi W_m$ and $\Wba_{n_W}/\Wti_{n_W}$ to simply $\Phi W$.
\end{defi}
Using diagram $(\dag)$ again we see that $\Phi W_m$ and $\Phi_m$ have the same $\ZM_p$-ranks, namely $s^+_m-1$, and of course that the natural action of $\Ga_{n_W}$ on all $\Phi W_m$'s is trivial.
\begin{theo}\label{CoW} For all $m\geq n \geq n_W$ we have isomorphisms~:
$$\Hha^0(G_{m,n}, \Wba_m)\cong \Phi W/p^{m-n}$$
and (group split) exact sequences~:
$$\xymatrix{ 0\ar[r] & (\ZM/p^{m-n})^{s^+} \ar[r] & H^{1}(G_{m,n},\Wba_m) \ar[r] & \Phi W[p^{m-n}] \ar[r] & 0}$$
\end{theo}
\begin{proof} Using lemma \ref{Wasy} the proof of this theorem is the same as the proof of theorem \ref{asy}. It is even actually easier because when splitting in two pieces the exact hexagone associated to the sequence $0\longrightarrow \Wti_m\longrightarrow \Wba_m \longrightarrow \Phi W \longrightarrow 0$, we only have to prove the injectivity of the map corresponding to $\varphi_{-1}$ by virtue of the triviality of $\Hha^0(G_{m,n},\Wti_m)$. This injectivity in turn follows from the stabilization of $\Phi W_m$ with respect to extension maps.
\end{proof}
\begin{coro} Let $\overset \rightarrow {W}_\infty$ be the inductive limit with respect to extension maps of the $\Wba_m$. Let $\tor(\Phi W)$ be the $\ZM_p$-torsion of $\Phi W$. For all $n$ greater than or equal to $n_W$ we have (group) isomorphisms~:
$$H^1(\Ga_n,\overset \rightarrow {W}_\infty)\simeq (\QM_p/\ZM_p)^{s^+}\oplus \tor(\Phi W), $$
and
$$H^2(\Ga_n,\overset \rightarrow {W}_\infty)\simeq (\QM_p/\ZM_p)^{s^+-1}.$$
\end{coro}
\begin{proof} The proof is exactly the same as the one of corollary \ref{infC}.
\end{proof}
\section{Universal co-norms of circular units}\label{UCN}
In this section we want to investigate the module $\Phi_m=\C_m/\Cti_m$ which, together with $KN_\infty$, describes the cohomology of the $\C_m$'s. By lemma \ref{Philem} we already know that its $\ZM_p$-rank is $s_m^+-1$ where $s_m^+$ is the number of places of $F_m^+$ dividing $p$. As we have $\Phi_m=\C_m/\Cti_m=\C_m^+/\Cti_m^+$ we may suppose without loss of generality that $F$ is totally real. Next, as the rank of $\Phi_m$ is known, we should concentrate on its $\ZM_p$-torsion, let's say $\tor(\Phi_m)$. To give examples of non trivial $\tor(\Phi_m)$ we state and prove a lemma which provides also criteria for triviality of $\tor(\Phi_m)$.
\begin{lemm}\label{exPhi} Recall that $I_m$ is the maximal subfield of $F_m$ such that $p$ does not ramify in $I_m/\QM$ and that $\overline{\mathrm{Cyc}}(I_m)$ is the $p$-completion of its module of cyclotomic numbers. Let $\si_p$ be the Fr\" obenius automorphism of $I_m$.
\begin{enumerate}
\item We have an exact sequence of finite groups
$$ 0\longrightarrow \Cti_m\cap \C(I_m) /\overline{\mathrm{Cyc}}(I_m)^{\si_p-1} \longrightarrow \Hha^{-1}(\langle \si_p \rangle, \overline{\mathrm{Cyc}}(I_m))\longrightarrow \tor(\Phi_m) \longrightarrow 0 .$$
\item Assume that $KN_\infty=0$ and  that $p$ does not ramify in $F$, then we have an isomorphism $\tor(\Phi_m)\cong \Hha^{-1}(\langle \si_p \rangle, \overline{\mathrm{Cyc}}(I_m))$
\end{enumerate}
\end{lemm}
\begin{proof} By lemma \ref{cond2} we have isomorphism $\Phi_m\cong C(I_m)/C(I_m)\cap \Cti_m$.
Using distribution relations it is easy to check that $\overline{\mathrm{Cyc}}(I_m)^{\si_p-1} \subset \C(I_m)\cap \Cti_m$. By mere rank computation the quotient $C(I_m)\cap \Cti_m / \mathrm{Cyc}(I_m)^{\si_p-1}$ is finite. Let $D_m$ be the maximal  subfield of $F_m$ such that $p$ (totally) splits in $D_m/\QM$, and let $N_{I_m/D_m}$ be the norm map. Write $\overline{\mathrm{Cyc}}(I_m)[N_{I_m/D_m}]$ for the kernel of this norm on $\overline{\mathrm{Cyc}}(I_m)$. Note that $\overline{\mathrm{Cyc}}(I_m)[N_{I_m/D_m}]\subset \C(I_m)$. Then $\overline{\mathrm{Cyc}}(I_m)[N_{I_m/D_m}]$ contains $\mathrm{Cyc}(I_m)^{\si_p-1}$ with finite index and is maximal with respect to that property inside $\C(I_m)$ because $N_{I_m/D_M} (\overline{\mathrm{Cyc}}(I_m)) \subset \overline{\mathrm{Cyc}}(D_m)$ is torsion-free. To summarize, we have inclusions of modules with finite index~:
\begin{eqnarray}\label{incs} \overline{\mathrm{Cyc}}(I_m)^{\si_p-1} \subset \Cti_m\cap \C(I_m) \subset \overline{\mathrm{Cyc}}(I_m)[N_{I_m/D_m}].\end{eqnarray}
It follows that $\tor(\Phi_m) = \overline{\mathrm{Cyc}}(I_m)[N_{I_m/D_m}]/\C(I_m)\cap \Cti_m$ and the exact sequence in 1 becomes tautological.
Moreover, if we assume that $p$ does not ramify in $F$ and $KN_\infty=0$, then we get $\Cti_m\cap \C(I_m)=
\Cti_m\cap\C(F)$ and $n_{U'C}=0$. By theorem \ref{CoCti} the sequence of modules $(\Cti_m)_m$ satisfies Galois descent and therefore $\Cti_m\cap \C(I_m)=\Cti_0$. Now it is not difficult to see on a system of generators of $\overline{\mathrm{Cyc}}(I_m)$ that, in that case, distribution relations imply the equality $\Cti_0=\overline{\mathrm{Cyc}}(I_m)^{\si_p-1}$, so that 2 follows from 1.
\end{proof}
Using the exact sequence of 1, we see that $\tor(\Phi_m)$ turns out to be trivial as soon as $p$ does not divide the order of $\si_p$. Intuitively in the sequence of inclusions (\ref{incs}), any couple of power of $p$ should occur in some cases as a couple of indices. We prove something far less ambitious which provides the simplest example of non-trivial
$\tor(\Phi_m)$.
\begin{prop}\label{exa} Let $F$ be a real abelian field of degree $p$ such that $\si_p$ generates $\Gal(F/\QM)$.
Then we have $KN_\infty=0$ and for all $m$ isomorphisms $\Phi_0\overset \sim \longrightarrow \Phi_m$.
Moreover $\Phi_0$ is finite of order $p$ if and only if at least two distinct rational primes ramify in $F$ and is trivial if and only if a single prime $\ell\neq p$ ramifies in $F$.
\end{prop}
\begin{proof} In the present case we have $s_m^+=1$ for all $m$, so that $\Phi_m=\tor(\Phi_m)$. Then $\Gal(F/\QM)$ is cyclic and $p$ is unramified in $F$ so that $KN_\infty$ is trivial (see \cite{JNT1}). By 2 of lemma \ref{exPhi} we deduce
$\Phi_m=\Phi_0=\Hha^{-1}(\langle \si_p \rangle, \overline{\mathrm{Cyc}}(F))$. In this case the module ${\mathrm{Cyc}}(F)$ is generated by the single number $N_{\QM(\ze_f)/F} (1-\ze_f)$ where $f$ is the conductor of $F$. If $f$ is a single prime power then this module is Galois free and has trivial cohomology. Else this module
is isomorphic to $\ZM[\ze_p]$, with residue field $\FM_p$ as $H^{-1}(\ZM/p,\ZM[\ze_p])$.
\end{proof}
To conclude this section of examples, let us be completely explicit. Take $p=3$, $\ell_1=7$ and $\ell_2=13$. Observe that $3$ is not a third power modulo $7$ neither modulo $13$. Then
for $i=1$ or $2$ the field $\QM(\ze_{\ell_i})$ contains a cubic subfield (say $F_i$) which admits $\FM_{27}$ as residue field (in other word the ideal (3) remains prime in $F_i$). Let $L$ be the compositum $F_1 F_2$. Then $\Gal(L/\QM)$ is of type $[3,3]$ and therefore $L$ admits 4 cubic subfields. These subfields are the inertia subfield at $\ell_2$, which is $F_1$, the inertia subfield at $\ell_1$, which is $F_2$, the decomposition subfield at $3$ say $D$ and a fourth subfield $F$ which has conductor $f=\ell_1\ell_2$ and in which the ideal $(3)$ remains prime. In these very simple cases, using proposition \ref{exa}, we have trivial $\Phi_0$ for the fields $F_i$, an example of finite $\Phi_0$ of order $p$ for the field $F$ and an example of $\ZM_p$-free $\Phi_0$ of rank $2$ for the field $D$.
\backmatter

\bigskip

{\bf Acknowledgement :} Its my pleasure to thank Thong Nguyen Quang Do and Filippo Nuccio for careful reading
and helpful comments.

\bigskip

\def\Dbar{\leavevmode\lower.6ex\hbox to 0pt{\hskip-.23ex \accent"16\hss}D}
  \def\cfac#1{\ifmmode\setbox7\hbox{$\accent"5E#1$}\else
  \setbox7\hbox{\accent"5E#1}\penalty 10000\relax\fi\raise 1\ht7
  \hbox{\lower1.15ex\hbox to 1\wd7{\hss\accent"13\hss}}\penalty 10000
  \hskip-1\wd7\penalty 10000\box7}
  \def\cftil#1{\ifmmode\setbox7\hbox{$\accent"5E#1$}\else
  \setbox7\hbox{\accent"5E#1}\penalty 10000\relax\fi\raise 1\ht7
  \hbox{\lower1.15ex\hbox to 1\wd7{\hss\accent"7E\hss}}\penalty 10000
  \hskip-1\wd7\penalty 10000\box7}
\providecommand{\bysame}{\leavevmode\hbox to3em{\hrulefill}\thinspace}
\providecommand{\MR}{\relax\ifhmode\unskip\space\fi MR }
\providecommand{\MRhref}[2]{%
  \href{http://www.ams.org/mathscinet-getitem?mr=#1}{#2}
}
\providecommand{\href}[2]{#2}



\begin{thebibliography}{NSW}

\bibitem[B1]{JNT1}
J.-R. Belliard, \emph{Sur la structure galoisienne des unit\'es
  circulaires dans les $\mathbb{Z}_{p}$-extensions}, J. Number Theory
  \textbf{69} (1998), no.~1, 16--49.

\bibitem[B2]{pmb}
\bysame, \emph{Sous-modules d'unit\'es en th\'eorie d'{I}wasawa}, Th\'eorie des
  nombres, Ann\'ees 1998/2001, Publ. Math. UFR Sci. Tech. Besan\c con, 2002,
  p.~12.

\bibitem[B3]{CJM}
\bysame, \emph{Global units modulo circular units : descent without {I}wasawa's
  {M}ain {C}onjecture}, Canad. J. Math. (2007), to appear.

\bibitem[BN]{BN2}
J.-R. Belliard and T.
  Nguy{\cftil{e}}n-Quang-{\Dbar}{\cftil{o}}, \emph{On modified circular units
  and annihilation of real classes}, Nagoya Math. J. \textbf{177} (2005),
  77--115.

\bibitem[G]{Gree73}
R. Greenberg, \emph{On a certain {$l$}-adic representation}, Invent. Math.
  \textbf{21} (1973), 117--124.

\bibitem[K1]{Kim92}
J.~M. Kim, \emph{Cohomology groups of cyclotomic units}, J. Algebra
  \textbf{152} (1992), no.~2, 514--519.

\bibitem[K2]{Kim95}
\bysame, \emph{Units and cyclotomic units in {$\mathbb{Z}_{p}$}-extensions},
  Nagoya Math. J. \textbf{140} (1995), 101--116.

\bibitem[K3]{Kim99}
\bysame, \emph{Circular units in the {$\mathbb{Z}_{p}$}-extensions of real
  abelian fields of prime conductor}, Tohoku Math. J. (2) \textbf{51} (1999),
  no.~3, 305--313.

\bibitem[KO]{KO01}
J.~M. Kim and S.~I. Oh, \emph{Cohomology groups of circular units}, J.
  Korean Math. Soc. \textbf{38} (2001), no.~3, 623--631.

\bibitem[Ku{\v{c}}]{Kuc03}
R. Ku{\v{c}}era, \emph{A note on circular units in {${\mathbb
  Z}_p$}-extensions}, J. Th\'eor. Nombres Bordeaux \textbf{15} (2003), no.~1,
  223--229, Les XXII\`emes Journ\'ees Arithmetiques (Lille, 2001).

\bibitem[KN]{KN95}
R. Ku{\v{c}}era and J. Nekov{\'a}{\v{r}}, \emph{Cyclotomic units in
$\mathbf {Z}_p$-extensions}, J. Algebra \textbf{171} (1995), no.~2, 457--472.

\bibitem[Kuz1]{Ku72}
L.~V. Kuz$'$min, \emph{The {T}ate module of algebraic number fields}, Izv.
  Akad. Nauk SSSR Ser. Mat. \textbf{36} (1972), 267--327.

\bibitem[Kuz2]{Ku96}
\bysame, \emph{On formulas for the class number of real abelian fields}, Izv.
  Ross. Akad. Nauk Ser. Mat. \textbf{60} (1996), no.~4, 43--110.

\bibitem[NSW]{NSW} J. Neukirch, A. Schmidt  and K. Wingberg, \emph{Cohomology of number fields},
    Grundlehren der Mathematischen Wissenschaften, \textbf{323},
 Springer-Verlag, Berlin, 2000.

\bibitem[N]{NG06}
T. Nguy{\cftil{e}}n-Quang-{\Dbar}{\cftil{o}}, \emph{Sur la conjecture faible de {G}reenberg dans le
  cas ab\'elien {$p$}-d\'ecompos\'e}, Int. J. Number Theory \textbf{2} (2006),
  no.~1, 49--64.

\bibitem[NL]{BNL}
T. Nguy{\cftil{e}}n-Quang-{\Dbar}{\cftil{o}} and M.~Lescop, \emph{Iwasawa descent and co-descent for
  units modulo circular units}, Pure Appl. Math. Q. \textbf{2} (2006), no.~2,
  465--496, With an appendix by J.-R.\ Belliard.

\bibitem[S]{Si80}
W. Sinnott, \emph{On the {S}tickelberger ideal and the circular units of an
  abelian field}, Invent. Math. \textbf{62} (1980), no.~2, 181--234.

\bibitem[W]{Wa}
L.~C. Washington, \emph{Introduction to cyclotomic fields}, second ed.,
  Springer-Verlag, New York, 1997.

\end{thebibliography}
\end{document}